\documentclass{amsart}

\usepackage{amssymb, dsfont, mathtools} 
\usepackage{fullpage, graphicx} 
\usepackage{hyperref} 
\usepackage{enumitem} 
\usepackage{xcolor} 
\usepackage{thmtools} 
\usepackage{caption, subcaption} 
\usepackage{multirow}

\hypersetup{colorlinks=true, urlcolor=purple, citecolor=blue, linkcolor=red, pdfstartview={XYZ null null 0.90}}

\DeclareMathOperator{\PSL}{PSL}

\DeclareMathOperator{\PGL}{PGL}

\newcommand{\ZZ}{\ensuremath{\mathbb{Z}}}
\newcommand{\RR}{\ensuremath{\mathbb{R}}}
\newcommand{\QQ}{\ensuremath{\mathbb{Q}}}
\newcommand{\CC}{\ensuremath{\mathbb{C}}}
\newcommand{\cir}{\ensuremath{\mathcal{C}}}
\newcommand{\cpack}{\ensuremath{\mathcal{A}}}

\newcommand{\sm}[4]{\ensuremath{\left(\begin{smallmatrix} #1 & #2\\#3 & #4\end{smallmatrix}\right)}}
\newcommand{\smvec}[2]{\ensuremath{\left(\begin{smallmatrix} #1 \\#2 \end{smallmatrix}\right)}}
\newcommand{\lm}[4]{\ensuremath{\left(\begin{matrix} #1 & #2\\#3 & #4\end{matrix}\right)}}
\newcommand{\lmvec}[2]{\ensuremath{\left(\begin{matrix} #1 \\#2 \end{matrix}\right)}}
\newcommand{\dkron}[2]{\left(\dfrac{#1}{#2}\right)}
\newcommand{\kron}[2]{\left(\frac{#1}{#2}\right)}
\newcommand{\quartic}[2]{\left[ \frac{#1}{#2} \right]}
\newcommand{\tablestretch}[1]{
\begin{center}
    \renewcommand{\baselinestretch}{1.3}
		{\upshape
		#1
		}
    \renewcommand{\baselinestretch}{1}
\end{center}}

\newtheorem{theorem}{Theorem}[section]
\newtheorem{conjecture}[theorem]{Conjecture}
\newtheorem{corollary}[theorem]{Corollary}
\newtheorem{lemma}[theorem]{Lemma}
\newtheorem{proposition}[theorem]{Proposition}

\newtheorem{alemma}{Lemma}[section]

\theoremstyle{definition}

\newtheorem{definition}[theorem]{Definition}

\newtheorem{remark}[theorem]{Remark}

\newtheorem{aremark}[alemma]{Remark}
\newtheorem{adefinition}[alemma]{Definition}

\numberwithin{equation}{section}

\begin{document}
	
\title[Apollonian local-global is false]{The local-global conjecture for Apollonian circle packings is false}

\author{Summer Haag}
\address{University of Colorado Boulder, Boulder, Colorado, USA}
\email{summer.haag@colorado.edu}
\urladdr{https://math.colorado.edu/~suha3163/}

\author{Clyde Kertzer}
\address{University of Colorado Boulder, Boulder, Colorado, USA}
\email{clyde.kertzer@colorado.edu}
\urladdr{https://clyde-kertzer.github.io/}

\author{James Rickards}
\address{Saint Mary's University, Halifax, Nova Scotia, Canada}
\email{james.rickards@smu.ca}
\urladdr{https://jamesrickards-canada.github.io/}

\author{Katherine E. Stange}
\address{University of Colorado Boulder, Boulder, Colorado, USA}
\email{kstange@math.colorado.edu}
\urladdr{https://math.katestange.net/}

\date{\today}
\thanks{Stange is supported by NSF DMS-2401580. Rickards and Stange were supported by NSF-
CAREER CNS-1652238 (PI Katherine E. Stange).}
\subjclass[2020]{Primary: 
52C26, 
11D99, 
11-04; 
 Secondary:
20H10, 
11E12, 
11A15, 
11B99} 

\keywords{Apollonian circle packings, quadratic reciprocity, quartic reciprocity, local-global conjecture, thin groups}
\begin{abstract}
	In a primitive integral Apollonian circle packing, the curvatures that appear must fall into one of six or eight residue classes modulo 24. The local-global conjecture states that every sufficiently large integer in one of these residue classes will appear as a curvature in the packing. We prove that this conjecture is false for many packings, by proving that certain quadratic and quartic families are missed. The new obstructions are a property of the thin Apollonian group (and not its Zariski closure), and are a result of quadratic and quartic reciprocity, reminiscent of a Brauer-Manin obstruction. Based on computational evidence, we formulate a new conjecture.
\end{abstract}
\maketitle

\setcounter{tocdepth}{1}
\tableofcontents

\setcounter{section}{-1} 
\section{Authors' note}
The \href{https://doi.org/10.4007/annals.2024.200.2.6}{published version} of this article differs slightly from this version. A few proofs, minor results, tables, and figures were either shortened or removed for publication. In particular, a few labelled statements in this version are not present. Any such differences are noted in parentheticals, and generally labelled with capital letters (2.A) instead of numbers (2.6).

\section{Introduction}

Apollonian circle packings (Figure~\ref{fig:ACP}) have served as a quintessential example in the study of thin groups, alongside problems such as Zaremba's conjecture for continued fractions (see \cite{Kont13}). The central conjecture is that ``thin orbits'' in $\ZZ$, such as orbits of a thin group like the Apollonian group (defined below), satisfy a \emph{local-global} property, namely that they are subject to certain congruence restrictions (local), but otherwise should admit all sufficiently large integers (global). The Apollonian local-global conjecture is due to Graham-Lagarias-Mallows-Wilks-Yan \cite{GLMWY02} and Fuchs-Sanden \cite{FS11}.  Existing lower bounds for the number of integers appearing as curvatures rely on analytic methods, quadratic forms, and the spectral theory of graphs.  This has culminated in the theorem of Bourgain and Kontorovich that amongst the admissible curvatures for an Apollonian circle packing (those values not obstructed by congruence conditions), a set of density one does appear \cite{BK14}.
In this paper, we demonstrate that, for infinitely many (and perhaps most, in a suitable sense) Apollonian circle packings, the local-global conjecture is nevertheless false.

The new obstruction rules out certain power families (such as $n^2$) as curvatures in certain packings.  It has its source in quadratic and quartic reciprocity, making it reminiscent of a Brauer-Manin obstruction.  However, it is strictly a phenomenon of the thin group; its Zariski closure has no such obstruction.

\subsection{Apollonian circle packings}

A theorem often attributed to Apollonius of Perga states that, given three mutually tangent circles in the plane, there are exactly two ways to draw a fourth circle tangent to the first three. Starting with three such circles, we can add in the two solutions of Apollonius (sometimes called Soddy circles after an ode by the famous chemist), obtaining five. New triples appear, and by continuing this process, one obtains a fractal called an \emph{Apollonian circle packing} (Figure~\ref{fig:ACP}). 

\begin{figure}[htb]
	\includegraphics{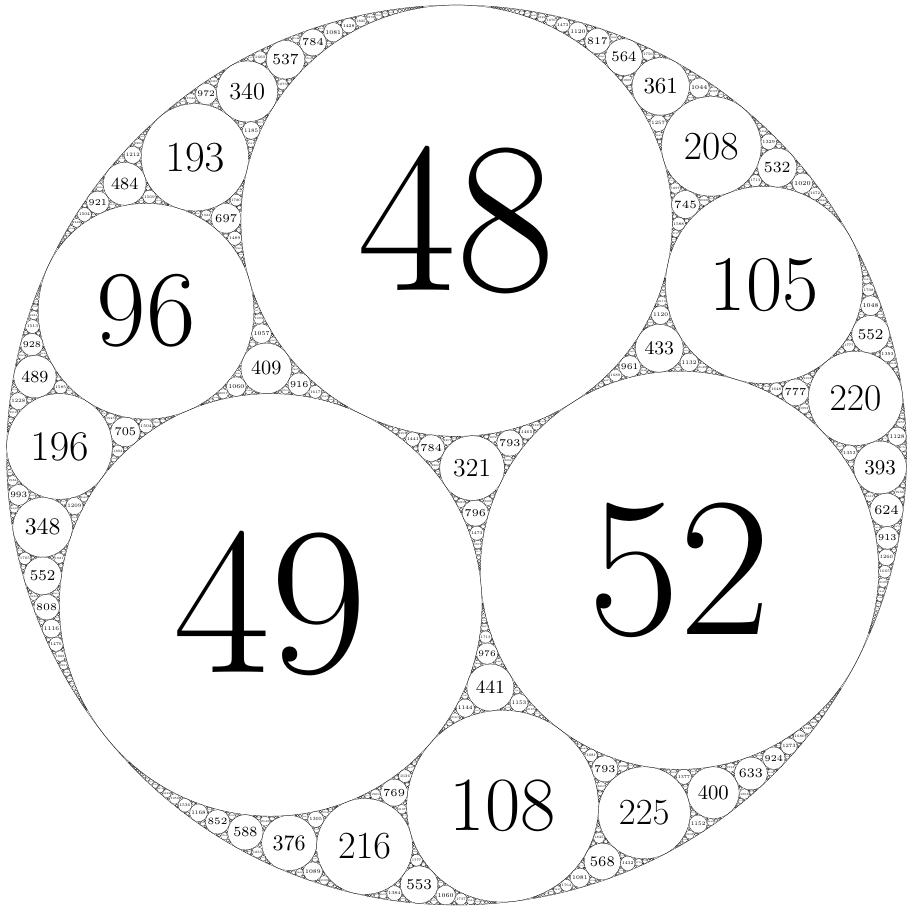}
	\caption{Circles of curvature $\leq 15000$ in the Apollonian circle packing corresponding to $(-23, 48, 49, 52)$. The local-global conjecture is false for every residue class modulo 24 in this packing.}\label{fig:ACP}
\end{figure}

Any solution to Apollonius' problem produces four mutually tangent circles.   
A \emph{Descartes quadruple} is a quadruple of four real numbers $(a, b, c, d)$ satisfying the \emph{Descartes equation} $(a+b+c+d)^2=2(a^2+b^2+c^2+d^2)$, $a+b+c+d>0$.
Given any Descartes quadruple, there exist four mutually tangent circles in the plane with those curvatures (straight lines are included as circles of curvature zero, and a negative curvatures represents a swap of orientation, placing the point at infinity in the interior).
In particular, a Descartes quadruple generates a unique Apollonian circle packing up to rigid motions (whereas an Apollonian circle packing contains many Descartes quadruples). For background, see \cite{GLMWY02}.

A simple but remarkable consequence of the Descartes equation is that if $a, b, c, d\in\ZZ$, then all curvatures in the packing are integral. Call such a configuration, and the packing it generates, \emph{integral}, and if we furthermore have $\gcd(a, b, c, d)=1$, call both \emph{primitive}. We frequently describe packings by the quadruple containing the four smallest curvatures, called the \emph{root quadruple}.

\subsection{The set of curvatures of an Apollonian circle packing}

For the rest of this paper, $\cpack$ will denote a primitive Apollonian circle packing. The study of the curvatures in $\cpack$ was first addressed in \cite[Section 6]{GLMWY02} as the ``Strong Density Conjecture.'' Later revised by Fuchs and Sanden in \cite{FS11}, it has come to be known as the ``local-global conjecture'' or ``local-global principle."
Call a positive curvature $c$ \emph{missing} in $\cpack$ if curvatures equivalent to $c\pmod{24}$ appear in $\cpack$ but $c$ does not.  

\begin{conjecture}[{\cite{GLMWY02,FS11}}]
\label{conj:old-local-global}
The number of missing curvatures in $\cpack$ is finite.
\end{conjecture}

As evidence toward the conjecture, the Hausdorff dimension of the Apollonian fractal dictates that in primitive packings, multiplicities increase as curvatures increase (see \cite{FS11} for data).  The current best known result toward the conjecture is due to Bourgain and Kontorovich.

\begin{theorem}[{\cite{BK14}}]
The number of missing curvatures up to $N$ is at most $O(N^{1-\eta})$ for some effectively computable $\eta > 0$.
\end{theorem}

In this paper, we prove that Conjecture \ref{conj:old-local-global} is false for many packings.

\begin{theorem}
	There exist infinitely many primitive Apollonian circle packings for which the number of missing curvatures up to $N$ is $\Omega(\sqrt{N})$. In particular, the local-global conjecture is false for these packings.
\end{theorem}

More precisely (Theorem~\ref{thm:mainthm}), certain quadratic and quartic families of curvatures (of the form $ux^2$ and $ux^4$ for a fixed integer $u$) are missing from some packings.

While Conjecture \ref{conj:old-local-global} is false in general, it may still hold for some packings.  Otherwise, we can account for the quadratic and quartic obstructions, and ask if the remaining set of missing curvatures is now finite.

\begin{definition}
	Define the \emph{sporadic set} $S_{\cpack}$ to be the set of missing curvatures that do not lie in one of the quadratic or quartic obstruction classes described in Theorem \ref{thm:mainthm}.
\end{definition}

We propose a replacement for Conjecture~\ref{conj:old-local-global}.

\begin{conjecture}\label{conj:newlocalglobal}
	The set $S_{\cpack}$ is finite.
\end{conjecture}

Using C and PARI/GP \cite{PARI}, we computed the sporadic set intersected with $[1, N]$ for various packings $\cpack$ and bounds $N$ in the range $[10^{10}, 10^{12}]$ (see GitHub repositories \cite{GHapollonian, GHmissingcurvatures} and Section~\ref{sec:computations}). The data appears to support Conjecture \ref{conj:newlocalglobal}.

\subsection{The method of proof}

To illustrate the method, we provide an example theorem.

\begin{theorem}
	The Apollonian circle packing $\cpack$ generated by the quadruple $(-3,5,8,8)$ has no square curvatures.
\end{theorem}

\begin{proof}
Fix  $\cir \in \cpack$ of curvature $n$. 
All curvatures in $\cpack$ are $0,1\pmod{4}$ by the congruence obstruction classification in Proposition~\ref{prop:admissibleresidues}.  It is well-known that the curvatures of the family of circles tangent to $\cir$ are the values properly represented (i.e. $\gcd(x, y)=1$) by a translated quadratic form $f_{\cir}(x,y) - n$ of discriminant $-4n^2$ (Section~\ref{sec:quadratic-forms}).  
Modulo $n$, the form $f_\cir$ becomes degenerate, being equivalent to $Ax^2$ for some coefficient $A$.  Since the numerator of the Kronecker symbol $\kron{ \cdot }{n}$ is invariant up to multiplication by squares and translation by $n$, the invertible values of $f_\cir$ reside in a single multiplicative coset of the squares (more detail is provided in Proposition~\ref{prop:singleton}). We can therefore define the symbol $\chi_2(\cir)$ to be the unique non-zero value of the Kronecker symbol $\kron{c}{n}$ as $c$ ranges over the curvatures of the circles tangent to $\cir$.

Let $\cir_1, \cir_2 \in \cpack$ be tangent, having non-zero coprime curvatures $a$ and $b$ respectively. Using quadratic reciprocity of the Kronecker symbol,
\[\chi_2(\cir_1)\chi_2(\cir_2)=\dkron{a}{b}\dkron{b}{a}=1.\]
In particular, $\chi_2(\cir_1) = \chi_2(\cir_2)$.

It can be shown that any two circles in $\cpack$ are connected by a path of consecutively coprime curvatures (Corollary~\ref{cor:coprimepath}).  Therefore, $\chi_2(\cir)$ is independent of the choice of circle $\cir \in \cpack$, and we have defined $\chi_2(\cpack)$.  Using the root quadruple, we compute
\[\chi_2(\cpack)=\dkron{8}{5} = -1.\]
We conclude that if there exists a circle $\cir$ of square curvature in $\cpack$, it would give $\chi_2(\cpack) = 1$, a contradiction. Therefore there do not exist square curvatures in $\cpack$.
\end{proof}

For general quadratic obstructions, we must modify the definition of $\chi_2(\cir)$ for technical reasons. 
The quartic obstructions are not dissimilar in spirit:  instead of studying the quadratic form of tangent curvatures,  we study the associated fractional ideal of $\ZZ[i]$, a perspective espoused in \cite{KS18}.

\subsection{Reciprocity obstructions in thin groups}
Call the quadratic and quartic obstructions found in this paper \emph{reciprocity obstructions}.  These are a feature of the orbit of the Apollonian group (i.e. all Descartes quadruples in a fixed packing), and not the orbit of the larger Zariski closure, which is an orthogonal group of transformations preserving the Descartes equation (the Zariski closure was found in \cite{Fuchs11}). Indeed, the super-Apollonian group lies between both groups, and an orbit contains all primitive integral Apollonian circle packings (see \cite{GLMWY06}), which contain all integers.  Sarnak also remarked on the lack of a spin group obstruction for the Descartes form in \cite[p. 301-302]{SarnakMonthly}.  

In the more general setting of thin groups, some orbits may not have reciprocity obstructions, or even congruence obstructions. For example, in \cite{Kont13}, it is shown that there are no congruence obstructions for Zaremba's conjecture. Furthermore, in \cite{Kont19}, it is shown that the local-global conjecture \emph{does} hold for Soddy sphere packings.  Despite this, one expects other instances of thin groups or semigroups to produce reciprocity obstructions; another example is studied in the follow-up paper \cite{RickardsStangeTwo}.

\subsection*{Acknowledgements}  The authors are grateful to the Department of Mathematics at the University of Colorado Boulder for sponsoring the Research Experience for Undergraduates and Graduates in Summer 2023, which led to this project. (These obstructions were first observed in the context of another problem; see Corollary~\ref{cor:discovery}.)  We are also grateful to Alex Kontorovich, Peter Sarnak, Richard Evan Schwartz and the anonymous reviewers for feedback on an earlier version of this paper, and to Arjun Gandhi, Ilyas Salhi, and Matthew Litman.

\section{Precise statement of the results}
	
\begin{proposition}\label{prop:admissibleresidues}
 	Let $R(\cpack)$ be the set of residues modulo $24$ of the curvatures in $\cpack$.  Then $R(\cpack)$ is one of six possible sets, labelled by a \emph{type} as follows:
	\tablestretch{
		\begin{tabular}{|c|c|} 
			\hline
			Type      & $R(\cpack)$ \\ \hline
			$(6, 1)$  & $0, 1, 4, 9, 12, 16$\\ \hline
			$(6, 5)$  & $0, 5, 8, 12, 20, 21$\\ \hline
			$(6, 13)$ & $0, 4, 12, 13, 16, 21$\\ \hline
			$(6, 17)$ & $0, 8, 9, 12, 17, 20$\\ \hline
			$(8, 7)$  & $3, 6, 7, 10, 15, 18, 19, 22$\\ \hline
			$(8, 11)$ & $2, 3, 6, 11, 14, 15, 18, 23$\\ \hline 
		\end{tabular}
    }
\end{proposition}
The set $R(\cpack)$ is called the \emph{admissible set} of the packing. The type $(x, k)$ denotes that $R(\cpack)$ has cardinality $x$, and the smallest positive residue in $R(\cpack)$ coprime\footnote{It is helpful later to have an invertible residue on hand, e.g. proof of Proposition~\ref{prop:6robstructions}.} to $24$ is $k$.

\begin{definition}\label{def:obstructiondef}
	Let $S_{d,u}:=\{un^d:n\in\ZZ\}$, $u, d > 0$. We say that the set $S_{d,u}$ forms a \textit{reciprocity obstruction} to $\cpack$ if infinitely many elements of $S_{d,u}$ are admissible in $\cpack$ modulo 24, and yet no element of $S_{d,u}$ appears as a curvature in $\cpack$.
	If $d=2$, we call it a \textit{quadratic obstruction}, and if $d=4$, it is a \textit{quartic obstruction}.
\end{definition}

It is clear that if there exists a reciprocity obstruction for $\cpack$, then the local-global conjecture \ref{conj:old-local-global} cannot hold for $\cpack$, and more specifically, for any of the admissible residue classes intersecting $S_{d, u}$. 

We will show that there exists a function
\[\chi_2:\{\text{circles in $\cpack$}\}\rightarrow \{\pm 1\},\]
which is constant across $\cpack$. In particular, this gives a well defined value for $\chi_2(\cpack)$. Furthermore, there exists a function 
\[\chi_4:\{\text{circles in a packing $\cpack$ of type $(6, 1)$ or $(6, 17)$}\}\rightarrow \{1, i, -1, -i\},\]
which satisfies $\chi_4(\cir)^2=\chi_2(\cir)$, and is also constant across a packing, defining $\chi_4(\cpack)$.
 
The value of $\chi_2$ determines the quadratic obstructions, and $\chi_4$ the quartic ones.  Full definitions come in Sections \ref{sec:chi2def} and \ref{sec:chi4def}, but in the simplest cases, $\chi_2(\cpack) = \kron{b}{a}$ for any coprime pair of curvatures $a$ and $b$ of tangent circles in $\cpack$.
The definition of $\chi_4$ relies on a finer invariant using the quartic residue symbol.  
The constancy across a packing follows from quadratic and quartic reciprocity.

\begin{definition}
	The (extended) type of $\cpack$ is either the triple $(x, k, \chi_2)$ or $(x, k, \chi_2, \chi_4)$, where $\cpack$ has type $(x, k)$ and corresponding values of $\chi_2$ (and $\chi_4$, if relevant).
\end{definition}

\begin{theorem}\label{thm:mainthm}
	The type of $\cpack$ implies the existence of certain quadratic and quartic obstructions, as described by the following table (which also includes the list of residues modulo $24$ where Conjecture \ref{conj:old-local-global} is false, and those where it is still open):

    \tablestretch{
     \begin{tabular}{|l|c|c|c|c|} 
        \hline
        Type              & Quadratic Obstructions  & Quartic Obstructions    & \ref{conj:old-local-global} false & \ref{conj:old-local-global} open\\ \hline
        $(6, 1, 1, 1)$    &                         &                         &                                   & $0, 1, 4, 9, 12, 16$\\ \hline
        $(6, 1, 1, -1)$   &                         & $n^4, 4n^4, 9n^4, 36n^4$ & $0, 1, 4, 9, 12, 16$              & \\ \hline
        $(6, 1, -1)$      & $n^2, 2n^2, 3n^2, 6n^2$ &                         & $0, 1, 4, 9, 12, 16$              & \\ \hline
        $(6, 5, 1)$       & $2n^2, 3n^2$            &                         & $0, 8, 12$                        & $5, 20, 21$\\ \hline
        $(6, 5, -1)$      & $n^2, 6n^2$             &                         & $0, 12$                           & $5, 8, 20, 21$\\ \hline
        $(6, 13, 1)$      & $2n^2, 6n^2$            &                         & $0$                               & $4, 12, 13, 16, 21$\\ \hline
        $(6, 13, -1)$     & $n^2, 3n^2$             &                         & $0, 4, 12, 16$                    & $13, 21$\\ \hline
        $(6, 17, 1, 1)$   & $3n^2, 6n^2$            & $9n^4, 36n^4$           & $0, 9, 12$                        & $8, 17, 20$\\ \hline
        $(6, 17, 1, -1)$  & $3n^2, 6n^2$            & $n^4, 4n^4$             & $0, 9, 12$                        & $8, 17, 20$\\ \hline
        $(6, 17, -1)$     & $n^2, 2n^2$             &                         & $0, 8, 9, 12$                     & $17, 20$\\ \hline
        $(8, 7, 1)$       & $3n^2, 6n^2$            &                         & $3, 6$                            & $7, 10, 15, 18, 19, 22$\\ \hline
        $(8, 7, -1)$      & $2n^2$                  &                         & $18$                              & $3, 6, 7, 10, 15, 19, 22$\\ \hline
        $(8, 11, 1)$      &                         &                         &                                   & $2, 3, 6, 11, 14, 15, 18, 23$\\ \hline
        $(8, 11, -1)$     & $2n^2, 3n^2, 6n^2$      &                         & $2, 3, 6, 18$                     & $11, 14, 15, 23$\\ \hline
    \end{tabular}
    }
\end{theorem}
 
\begin{aremark}[not in published version]
	The intersection of quadratic and quartic obstructions with a residue class can be described by adding a condition on $n$. For example, the obstruction $2n^2$ in type $(6, 17, -1)$ intersects the class $8\pmod{24}$ as $2(6n\pm 2)^2$, and the class $0\pmod{24}$ as $2(6n)^2$.
\end{aremark}

\begin{remark}
	We could consider the $\chi_4$ value for packings of types $(6, 1, -1)$ and $(6, 17, -1)$, but the quartic obstructions in these cases are subsumed in the quadratic ones.
\end{remark}
	
\begin{corollary}
The local-global conjecture \ref{conj:old-local-global} is false for at least one residue class in all primitive Apollonian circle packings that are not of type $(6, 1, 1, 1)$ or $(8, 11, 1)$.
\end{corollary}

The exceptions where the local-global conjecture may yet hold include the strip packing (root quadruple $(0, 0, 1, 1)$), and the bug-eye packing (root quadruple $(-1, 2, 2, 3)$).

The following corollary is the phenomenon which led to the discovery of quadratic obstructions.

\begin{corollary}
\label{cor:discovery}
	Curvatures $24m^2$ {\upshape (}necessarily $0\bmod{24}${\upshape )} and $8n^2$ with $3\nmid n$ {\upshape (}necessarily $8\bmod{24}${\upshape )} cannot appear in the same primitive Apollonian circle packing, despite $0\pmod{24}$ and $8\pmod{24}$ being simultaneously admissible in packings of type $(6, 5)$ or $(6, 17)$.
\end{corollary}

\begin{remark}
Apollonian circle packings have been generalized in a variety of ways.  In \cite{KS18_Bi}, $K$-Apollonian packings were defined for each imaginary quadratic field $K$, where the $\QQ(i)$-Apollonian case is the subject of this paper. It is quite possible that quadratic obstructions occur in these packings, which share many features with the present case. The existence of quartic obstructions is less likely, as it arises because $K=\QQ(i)$.  The family of packings studied in \cite{FSZ19} are also governed by quadratic forms (this was the essential feature needed for the positive density results of that paper), and include the $K$-Apollonian packings; these are likely subject to quadratic obstructions as well.  It would also be interesting to ask the same question about an even wider class of packings studied by Kapovich and Kontorovich \cite{KK23}.
\end{remark}
	
In Section \ref{sec:background}, we cover background. Quadratic obstructions are addressed in Section \ref{sec:quadraticobstructions}, and quartic obstructions in Section \ref{sec:quarticbstructions}. Computational evidence to support Conjecture \ref{conj:newlocalglobal} is found in Section \ref{sec:computations}.

\section{Residue classes and quadratic forms}\label{sec:background}

Consider a Descartes quadruple $(a, b, c, d)$ contained in an Apollonian circle packing. The act of swapping the $i$\textsuperscript{th} circle to obtain the other solution described by Apollonius is called a \emph{move}, and is denoted $S_i$. In terms of the quadruples, $S_1$ corresponds to
\[S_1:(a, b, c, d)\rightarrow (2b+2c+2d-a, b, c, d).\]
Analogous equations for $S_2$ to $S_4$ hold. It is possible to move between every pair of Descartes quadruples in a fixed circle packing via a finite sequence of these moves (up to a permutation of the entries of the quadruples).  The classical Apollonian group is generated by these four moves as transformations in the orthogonal group preserving the Descartes form (which has signature $(3,1)$).

\subsection{Residue classes}

In \cite{GLMWY02}, Graham-Lagarias-Mallows-Wilks-Yan determined the list of possible residues modulo 12 in $\cpack$. In her Ph.D. thesis \cite{Fuchs10}, Fuchs proved that there are in fact restrictions modulo $24$ and that this is the ``best possible.'' The complete list of obstructions modulo $24$ are found in Proposition \ref{prop:admissibleresidues}, which is proven in this section.

\begin{proposition}\label{prop:aplusbrestriction}
	Consider tangent circles in $\cpack$ with curvatures $a, b$. Then $a+b\not\equiv 3, 6, 7\pmod{8}$. In particular, if $\cpack$ has type $(8, k)$ and $a,b$ are odd, then one is $3\pmod{8}$ and the other is $7\pmod{8}$.
\end{proposition}
\begin{proof}
	Assume otherwise. Then the odd part of $a+b$ is $3\pmod{4}$,  so there exists a prime $p\equiv 3\pmod{4}$ with $v_p(a+b)$ odd, where $v_p$ represents the $p$-adic valuation. Rearranging the Descartes equation gives
	\[(a-b)^2+(c-d)^2=2(a+b)(c+d).\]
	Since the left hand side is a sum of two squares that is a multiple of $p\equiv 3\pmod{4}$, it follows that $p\mid a-b, c-d$, and $v_p(\text{LHS})$ is even. Therefore $v_p(c+d)$ is odd, hence $p\mid c+d$ as well. Thus $p\mid a, b, c, d$, so the quadruple is not primitive, a contradiction.
    The final part follows immediately.
\end{proof}

Rather than work modulo 24, we consider modulo 3 and 8 separately.

\begin{alemma}[not in published version]\label{lem:mod3}
	The set of Descartes quadruples in $\cpack$ taken modulo $3$ is one of the following two sets:
	\begin{enumerate}[label=\upshape{(\alph*)}]
		\item $\{\text{\upshape{all permutations of }}(0, 0, 1, 1) \text{\upshape{ and }} (0, 1, 1, 1)\}$;
		\item $\{\text{\upshape{all permutations of }}(0, 0, 2, 2) \text{\upshape{ and }} (0, 2, 2, 2)\}$.
	\end{enumerate}
\end{alemma}
\begin{proof}
	By dividing into cases based on the number of curvatures that are multiples of 3, a straightforward computation shows the claimed sets are the only solutions to Descartes's equation modulo 3. By considering the moves $S_1$ to $S_4$, we see that they fall into the two classes.
\end{proof}

\begin{alemma}[not in published version]\label{lem:mod8}
	The set of Descartes quadruples in $\cpack$ taken modulo $8$ is one of the following three sets:
	\begin{enumerate}[label=\upshape{(\alph*)}]
		\item $\{\text{\upshape{all permutations of }}(0, 0, 1, 1), (0, 4, 1, 1)\text{\upshape{ and }} (4, 4, 1, 1)\}$;
		\item $\{\text{\upshape{all permutations of }}(0, 0, 5, 5), (0, 4, 5, 5),\text{\upshape{ and }} (4, 4, 5, 5)\}$;
		\item $\{\text{\upshape{all permutations of }}(2, 2, 3, 7), (2, 6, 3, 7),\text{\upshape{ and }} (6, 6, 3, 7)\}$.
	\end{enumerate}
\end{alemma}
\begin{proof}
	Let $(a, b, c, d)$ be a Descartes quadruple in $\cpack$. Considering the Descartes equation modulo 2, there is an even count of odd numbers amongst $a, b, c, d$. This cannot be zero (due to primitivity), and it cannot be four, as otherwise $(a+b+c+d)^2\equiv 8\pmod{16}$. Therefore, there are always two odd and two even curvatures, so without loss of generality, assume that $c, d$ are odd and $a, b$ are even.
	
	Assume $c\equiv 1\pmod{4}$. By Proposition \ref{prop:aplusbrestriction}, $a\equiv b\equiv 0\pmod{4}$. In turn, this implies $d\equiv 1\pmod{4}$ and $d\equiv c\pmod{8}$. This gives the quadruples listed in (a) and (b), and by applying $S_1$ through $S_4$ and permutations, we see them fall into the two classes.
	
	Otherwise, $c\equiv 3\pmod{4}$, which analogously implies $a\equiv b\equiv 2\pmod{4}$, $d\equiv 3\pmod{4}$, and $c\not\equiv d\pmod{8}$. This gives the quadruples in (c), and again, the moves $S_1$ to $S_4$ and permutations show that they form one class.
\end{proof}

\begin{aremark}[not in published version]
    A finer version of this question is to fix a Descartes quadruple $(a, b, c, d)$ and ask which quadruples modulo $n$ are obtainable from a sequence of the moves $S_1$ to $S_4$ (i.e. ignoring permutations). Lemma~\ref{lem:mod3} remains valid, but Lemma~\ref{lem:mod8} changes slightly. Each of the modulo 8 sets partitions into the six subsets where the even and odd curvatures remain in fixed places.
\end{aremark}

A consequence of Lemmas \ref{lem:mod3} and \ref{lem:mod8} is that
\begin{itemize}
	\item $R(\cpack)\pmod{3}=$ $\{0, 1\}$ or $\{0, 2\}$;
	\item $R(\cpack)\pmod{8}=$ $\{0, 1, 4\}$ or $\{0, 4, 5\}$ or $\{2, 3, 6, 7\}$.
\end{itemize}

The Chinese remainder theorem gives six ways to combine these into a congruence set modulo 24, resulting in the sets listed in Proposition \ref{prop:admissibleresidues}. For each of the six sets, there do exist primitive packings with those admissible sets. It remains to show that the Chinese remainder theorem holds for curvatures.

\begin{alemma}[not in published version]
	Let $\cpack$ contain a curvature equivalent to $r_1\pmod{3}$ and another curvature equivalent to $r_2\pmod{8}$. Then there exists a curvature in $\cpack$ that is simultaneously equivalent to $r_1\pmod{3}$ and $r_2\pmod{8}$.
\end{alemma}
\begin{proof}
	This is a special case of Lemma 4.4 of \cite{Fuchs10}, and can also be proven by direct computation.
\end{proof}

Interestingly, not all solutions to the Descartes equation modulo $24$ lift to solutions in $\ZZ$; the sum-of-squares argument in Proposition~\ref{prop:aplusbrestriction}  (essentially, an effect of quadratic reciprocity) rules some out. For example, $(0, 0, 1, 13)$ is a solution modulo 24 that does not lift since $1+13\equiv 6\pmod{8}$.

\subsection{Quadratic forms}
\label{sec:quadratic-forms}

See the books by Buell \cite{BU89} or Cohen \cite{Cohen93} for a longer exposition on quadratic forms.

\begin{adefinition}[not in published version]
	A primitive integral positive definite binary quadratic form is a function of the form $f(x, y)=Ax^2+Bxy+Cy^2$, where $A, B, C\in\ZZ$, $\gcd(A, B, C)=1$, $A>0$, and $D:=B^2-4AC<0$. Call $D$ the discriminant of $f$. The group $\PGL(2, \ZZ)$ acts on the set of forms as follows:
	\[\lm{a}{b}{c}{d}\circ f:=f(ax+by, cx+d).\]
	This action preserves the discriminant, and divides the set of forms of a fixed discriminant into a finite number of equivalence classes (the natural group structure obtained by taking $\PSL(2, \ZZ)$ equivalence classes does not extend to the $\PGL(2, \ZZ)$ equivalence classes).
\end{adefinition}

The connection between Descartes quadruples and quadratic forms dates to \cite{GLMWY02} and \cite{Sar07}. Fixing a circle $\cir$, one can associate to it a certain quadratic form $f_\cir$.  Then, in particular, the circles tangent to $\cir$ have curvatures properly represented by a shift of $f_\cir$.

\begin{proposition}[{\cite[Theorem 4.2]{GLMWY02}}]\label{prop:quadformbij}
	There is a bijection between primitive integral positive definite binary quadratic forms of discriminant $-4a^2$ and primitive Descartes quadruples containing $a$ as the first curvature. The map from quadruple to form is
    \[
    (a,b,c,d) \mapsto (a+b)x^2+(a+b+c-d)xy+(a+c)y^2,
    \]
    and the map from form to quadruple is
    \[
    Ax^2 + Bxy + Cy^2 \mapsto (a, A-a, C-a, A+C-B-a).
    \]
\end{proposition}

Given a circle $\cir$ of curvature $a$ in a primitive Apollonian circle packing, we can associate a Descartes quadruple $(a, b, c, d)$ to $\cir$, and therefore a quadratic form. The ambiguity in choosing the Descartes quadruple exactly corresponds to taking a $\PGL(2, \ZZ)-$equivalence class of quadratic forms. See \cite{GLMWY02} and \cite[Proposition 3.1.3]{JR23stairs}.

\begin{definition}
	Given $\cir\in\cpack$ a circle of curvature $n$, define $f_{\cir}$ to be a quadratic form of discriminant $-4n^2$ that corresponds to $\cir$ via Proposition \ref{prop:quadformbij}.
\end{definition}

For the rest of the paper we will assume that $n\neq 0$ for convenience. The results should still hold for $n=0$, but as this only corresponds to the strip packing $(0, 0, 1, 1)$, it will be of no use here.

In \cite{Sar07}, Sarnak made a crucial observation relating curvatures of circles tangent to $\cir$ and properly represented values of $f_{\cir}$, which is a key tool for the rest of the paper.

\begin{proposition}
	Let $\cir$ be a circle of curvature $n$ in $\cpack$.
	The multiset of curvatures of circles tangent to $\cir$ in $\cpack$ is $\{ f_{\cir}(x, y)-n: \gcd(x,y) = 1 \}$.
\end{proposition}

\section{Quadratic obstructions}\label{sec:quadraticobstructions}
Let $u\in\{1, 2, 3, 6\}$. The strategy to prove that no element of $S_{2,u}=\{uw^2:w\in\ZZ\}$ appears as a curvature in $\cpack$ is
\begin{enumerate}
    \item For each circle $\cir\in\cpack$, define  $\chi_2(\cir) \in \{ \pm 1 \}$
    and demonstrate that it is an invariant of $\cpack$:
    \begin{enumerate}
    	\item The value of $\chi_2$ is equal for tangent circles with coprime curvatures.
    	\item One can walk between any two circles via coprime tangencies.
    \end{enumerate}
    \item Packings with a certain $\chi_2(\cpack)$ value and type cannot accommodate curvatures from $S_{2,u}$ amongst circles tangent to a ``large'' subset of $\cpack$.
	\item Every circle in $\cpack$ is tangent to a circle in this large subset.
\end{enumerate}

\subsection{Definition of $\chi_2$}\label{sec:chi2def}
Let $f(x, y)=Ax^2+Bxy+Cy^2$ be a primitive integral positive definite binary quadratic form of discriminant $-4n^2$ for an integer $n\neq 0$.

\begin{proposition}\label{prop:singleton}
	There is a unique properly represented and invertible residue $f(x,y)$ modulo $n$, up to multiplication by a square.
\end{proposition}

Denote any lift of this value to the positive integers by $\rho(f)$.

\begin{proof}
	If $A$ is coprime to $n$, observe that
	\[f(x, y)=A\left(x+\dfrac{B}{2A}y\right)^2+\dfrac{n^2}{A}y^2\equiv A\left(x+\dfrac{B}{2A}y\right)^2\pmod{n},\]
	hence this uniquely lies in the coset containing $A$. If $A$ is not coprime to $n$, replace $f$ by an appropriate $\PSL(2, \ZZ)$-translate of $f$.
\end{proof}

This allows us to give a condition for numbers that are not represented by $f$. First, let us recall a few basic properties of the Kronecker symbol. If $a,b,n\in\ZZ^{\geq 0}$, then
\begin{itemize}
    \item $\kron{ab}{n}=\kron{a}{n}\kron{b}{n}$.
    
    \item $\kron{a}{n}=\kron{b}{n}$ if $a\equiv b\pmod{n}$ and $n\not\equiv 2\pmod{4}$, or $a\equiv b\pmod{4n}$ and $n\equiv 2\pmod{4}$.

    \item (Quadratic reciprocity) Write $a=2^ea^{\circ}$ and $b=2^fb^{\circ}$ where $a^{\circ},b^{\circ}$ are odd ($0^{\circ}=1$ by convention). Then
    \[\kron{a}{b}=(-1)^{\frac{a^{\circ}-1}{2}\frac{b^{\circ}-1}{2}}\kron{b}{a}.\]
\end{itemize}

\begin{proposition}\label{prop:quadnosoln}
	Let $n'=\frac{n}{2}$ if $n\equiv 2\pmod{4}$ and $n'=n$ otherwise. Then the Kronecker symbol $\kron{\rho(f)}{n'}$ is independent of the choice of $\rho(f)$ and takes values in $\{\pm 1\}$. Furthermore, let $uw^2$ be coprime to $n$, where $u$ and $w$ are positive integers. If $\kron{\rho(f)}{n'}\neq\kron{u}{n'}$, then $f(x, y)$ does not properly represent $n+uw^2$.
\end{proposition}
\begin{proof}
	Let $\rho_1$ and $\rho_2$ be two choices for $\rho(f)$. Then there exist integers $s, t$ such that $\gcd(s, n)=1$ and $\rho_1=s^2\rho_2+tn$.	Since $n'\not\equiv 2\pmod{4}$ and $\rho_1, \rho_2> 0$,
	\[\dkron{\rho_1}{n'}=\dkron{s^2\rho_2+tn}{n'}=\dkron{s^2\rho_2}{n'}=\dkron{\rho_2}{n'},\]
	hence $\kron{\rho(f)}{n'}$ is well-defined. As $\rho(f)$ is coprime to $n'$, the symbol takes values in $\{\pm 1\}$.
	
	Finally, if $f(x, y)$ properly represents $n+uw^2$, we can take $\rho(f)=u$, and the result follows from above.
\end{proof}

By using the correspondence between circles and quadratic forms, we can now assign a sign $\pm 1$ to each circle in an Apollonian circle packing, which will dictate what quadratic obstructions must occur adjacent to it. By Proposition \ref{prop:quadnosoln}, the following is well-defined.

\begin{definition}\label{def:chi}
	Let $\cir\in\cpack$ be a circle of curvature $n$, and let $\rho =\rho(f_{\cir})$. Define
	\[\chi_2(\cir):=\begin{dcases}
		\dkron{\rho}{n}       & \text{if $n\equiv 0, 1\pmod{4}$;}\\
		\dkron{-\rho}{n/2} & \text{if $n\equiv 2\pmod{4}$;}\\
		\dkron{2\rho}{n}      & \text{if $n\equiv 3\pmod{4}$.}
	\end{dcases}\]
\end{definition}

\subsection{Propagation of $\chi_2$}
We now show that the value $\chi_2(\cir)$ is constant (propagates) across the packing $\cpack$.

\begin{proposition}\label{prop:1step}
	Let $\cir_1,\cir_2\in\cpack$ be tangent circles with coprime curvatures. Then $\chi_2(\cir_1)=\chi_2(\cir_2)$.
\end{proposition}
\begin{proof}

	Let the curvatures of $\cir_1,\cir_2$ be $a,b$ respectively. Since $\gcd(a+b, a)=\gcd(a+b, b)=1$, by Proposition \ref{prop:quadformbij} we can take $\rho(f_{\cir_1})=a+b=\rho(f_{\cir_2})$ (noting that $a+b>0$).
	
	First, assume the packing has type $(6, k)$. Then $a,b\equiv 0, 1\pmod{4}$, with at least one being odd. Therefore
	\[\chi_2(\cir_1)\chi_2(\cir_2)=\dkron{a+b}{a}\dkron{a+b}{b}=\dkron{a}{b}\dkron{b}{a}=1,\]
	by quadratic reciprocity. This implies that $\chi_2(\cir_1)=\chi_2(\cir_2)$, as claimed.
	
	Otherwise, the packing has type $(8, k)$, and we make two cases. If both $a$ and $b$ are odd, by Proposition~\ref{prop:aplusbrestriction} we can assume $a\equiv 3\pmod{8}$ and $b\equiv 7\pmod{8}$. Thus 
	\[\chi_2(\cir_1)\chi_2(\cir_2)=\dkron{2a+2b}{a}\dkron{2a+2b}{b}=\dkron{2}{ab}\dkron{b}{a}\dkron{a}{b}=(-1)(-1)=1,\]
	where we used the fact that $ab\equiv 5\pmod{8}$ and quadratic reciprocity.
	
	Finally, assume that $a$ is odd and $b$ is even, necessarily $2\pmod{4}$. We write $b=2b'$ and compute
	\[\chi_2(\cir_1)\chi_2(\cir_2)=\dkron{2a+2b}{a}\dkron{-a-b}{b'}=\dkron{4b'}{a}\dkron{-a}{b'}=\dkron{-1}{b'}\dkron{b'}{a}\dkron{a}{b'}=(-1)^{(b'-1)/2}(-1)^{(b'-1)/2}=1,\]
	completing the proof.
\end{proof}

\subsection{Coprime curvatures}

The following lemma is a straightforward consequence of Proposition \ref{prop:quadformbij}.

\begin{lemma}\label{lemma:crackfamily}
	Let $(a, b, c, d)$ be a Descartes quadruple, where curvatures $a, b$ correspond to circles $\cir_1,\cir_2$ respectively. The curvatures of the family of circles tangent to both $\cir_1$ and $\cir_2$ are parameterized by
    \[f(x) = (a + b)x^2 - (a + b + c - d)x + c, \quad x \in \ZZ.\]
\end{lemma}
\begin{proof}
    The relevant family of curvatures is given by applying powers of $S_4S_3$ to the initial quadruple. A two-tailed induction yields
    \[(S_4S_3)^k(a,b,f(0),f(1))=(a,b,f(2k),f(2k+1)),\quad k \in \ZZ,\]
    giving the result.
\end{proof}

Using this family, we can add in extra circles in a tangency walk to ensure coprimality.

\begin{lemma}\label{lem:makecoprime}
	Let $\cir_1, \cir_2\in \cpack$ be tangent circles of respective curvatures $a, b$. Then there exists a circle $\cir'$ that is tangent to both $\cir_1$ and $\cir_2$ whose curvature is coprime to both $\cir_1$ and $\cir_2$.
\end{lemma}
\begin{proof}
  It suffices to show that for every prime $p\mid ab$, the function $f(x)$ of Lemma~\ref{lemma:crackfamily} takes a value coprime to $p$. Since primitive Descartes quadruples contain two odd and two even numbers, the result follows for $p=2$. If $p>2$, a quadratic polynomial will completely vanish modulo $p$ only if its coefficients vanish. This implies $p\mid a, b, c, d$, a contradiction in a primitive packing.
\end{proof}

The following corollary is an immediate consequence of Lemma \ref{lem:makecoprime}.

\begin{corollary}\label{cor:coprimepath}
	Let $\cir, \cir'\in \cpack$. Then there exists a path of tangent circles from $\cir$ to $\cir'$ where consecutive circles have coprime curvature.
\end{corollary}

The invariance of $\chi_2(\cir)$ follows directly from Proposition \ref{prop:1step} and Corollary \ref{cor:coprimepath}.

\begin{corollary}\label{cor:constantchi}
	The value of $\chi_2$ is constant across all circles in a fixed primitive Apollonian circle packing $\cpack$. Denote this value by $\chi_2(\cpack)$.
\end{corollary}

Before proving the sets of quadratic obstructions, we have one final coprimality lemma.

\begin{lemma}\label{lem:coprimetangent}
	Let $\cir$ have curvature $n$. Then there exists a circle tangent to $\cir$ with curvature coprime to $6n$.
\end{lemma}

\begin{proof}
    By considering the possible Descartes quadruples modulo 2 and 3, it can be shown that $\cir$ is tangent to a circle $\cir'$ with curvature $m$ where $6\mid nm$. The result now follows from Lemma~\ref{lem:makecoprime}.
\end{proof}

\subsection{Quadratic obstructions}
For each type of packing, we can assemble the above results to determine which values of $u$ and $\chi_2$ produce quadratic obstructions.

\begin{proposition}\label{prop:6robstructions}
	Let $\cpack$ have type $(6, k)$. Then the following quadratic obstructions occur, as a function of type and $\chi_2(\cpack)$:
	\tablestretch{
    \begin{tabular}{|c|c|c|} 
        \hline
        Type    & $\chi_2(\cpack)$ & Quadratic obstructions\\ \hline
        (6, 1)  & 1              & \\ 
                & -1             & $n^2, 2n^2, 3n^2, 6n^2$\\ \hline
        (6, 5)  & 1              & $2n^2, 3n^2$\\
                & -1             & $n^2, 6n^2$\\ \hline
        (6, 13) & 1              & $2n^2, 6n^2$\\
                & -1             & $n^2, 3n^2$\\ \hline
        (6, 17) & 1              & $3n^2, 6n^2$\\
                & -1             & $n^2, 2n^2$\\ \hline
    \end{tabular}
    }
\end{proposition}
\begin{proof}
    Assume that a circle of curvature $uw^2$ appears in a packing $\cpack$ of type $(6,k)$. By Lemma \ref{lem:coprimetangent}, it is tangent to a circle $\cir$ with curvature $n$ coprime to $6uw^2$, hence $n\equiv k\pmod{24}$. By Proposition \ref{prop:quadnosoln}, the existence of the curvature $uw^2$ tangent to $\cir$ implies that 
	\[\dkron{u}{n}=\kron{\rho(f_{\cir})}{n}=\chi_2(\cpack),\]
	using Definition \ref{def:chi} and Corollary \ref{cor:constantchi}. By quadratic reciprocity,
	\tablestretch{
			\begin{tabular}{|c|c|c|} 
				\hline
				$k\pmod{24}$ & $\kron{2}{n}$ & $\kron{3}{n}$\\ \hline
				1  & 1  & 1  \\ \hline
				5  & -1 & -1 \\ \hline
				13 & -1 & 1  \\ \hline
				17 & 1  & -1 \\ \hline
			\end{tabular}
	}
	These values give the claimed table.
\end{proof}

The $(6, k)$ entries in the table of Theorem~\ref{thm:mainthm} are filled in by intersecting the quadratic obstructions with the possible residue classes. Note that not listing a value of $u$ as a quadratic obstruction in the table does not imply that it cannot be an obstruction, only that this proof method does not rule it out. The completeness of these lists is discussed in Section \ref{sec:computations}.

\begin{proposition}
	Let $\cpack$ have type $(8, k)$.  Then the following quadratic obstructions occur, as a function of type and $\chi_2(\cpack)$:
	\tablestretch{
        \begin{tabular}{|c|c|c|} 
            \hline
            Type    & $\chi_2(\cpack)$ & Quadratic obstructions\\ \hline
            (8, 7)  & 1              & $3n^2, 6n^2$\\ 
                    & $-1$           & $2n^2$   \\ \hline
            (8, 11) & 1              &      \\
                  & $-1$           & $2n^2, 3n^2, 6n^2$\\ \hline
        \end{tabular}
    }
\end{proposition}
\begin{proof}
    Repeat the proof of Proposition \ref{prop:6robstructions}: assume that a circle of curvature $uw^2$ appears in $\cpack$. Then there exists a circle of curvature $n$ coprime to $6uw^2$ that is tangent to our starting circle. Thus
	\[\dkron{u}{n}=\kron{\rho(f_{\cir})}{n}=\dkron{2}{n}\chi_2(\cpack),\]
	giving $\chi_2(\cpack)= \kron{2u}{n}$. By quadratic reciprocity,
	\tablestretch{
    \begin{tabular}{|c|c|c|} 
        \hline
        $k\pmod{24}$ & $\kron{2}{n}$ & $\kron{3}{n}$\\ \hline
        7  & 1  & -1 \\ \hline
        11 & -1 & 1  \\ \hline
        19 & -1 & -1 \\ \hline
        23 & 1  & 1  \\ \hline
    \end{tabular}
	}
 
	First, assume $u=2$. Then $\chi_2(\cpack)=1$, giving the conditions.

	Next, take $u=3$. The only admissible residue class with elements of the form $3w^2$ is $3\pmod{24}$, so we can assume that $w$ is odd. If the packing has type $(8, 7)$, we divide in two cases. 
	
	If $n\equiv 7\pmod{24}$, then $\chi_2(\cpack) = \kron{6}{n}=-1$. The other possibility is $n\equiv 19\pmod{24}$. In this case the two circles under consideration have curvatures which are $3 \pmod{8}$, in contradiction to Proposition \ref{prop:aplusbrestriction}.
	
	For packing type $(8, 11)$, a similar argument works.  The case $n\equiv 11\pmod{24}$ is ruled out by Proposition \ref{prop:aplusbrestriction}, and $n\equiv 23\pmod{24}$ implies $\chi_2(\cpack)=\kron{6}{n}=1$.
	
	The final case is $u=6$. Then $\chi_2(\cpack) = \kron{3}{n}$, and the conclusion follows from quadratic reciprocity.
\end{proof}

\begin{remark}
	It is reasonable to ask if the results in this section can be extended to other values of $u$. The proof will work for larger values of $u$ that have no prime divisors other than $2$ or $3$, but these obstructions are already contained in those with $u\mid 6$. If $u$ has a prime divisor $p\geq 5$, then the Kronecker symbol $\kron{p}{n}$ is not uniquely determined from $n\pmod{24}$. It will rule out $uw^2$ from appearing tangent to a subset of the circles in $\cpack$, but this is not enough to cover the entire packing. Interestingly, this suggests that there may be ``partial'' obstructions: quadratic families whose members appear less frequently than other curvatures of the same general size.
\end{remark}

\section{Quartic obstructions}\label{sec:quarticbstructions}
The proof strategy in this section is similar to the quadratic case, where we define an invariant on the circles in the packing, and show that this forbids certain curvatures. The main difference is the quartic restrictions will only apply to two types of packings, and the propagation of the invariant comes down to quartic reciprocity for $\ZZ[i]$.

\subsection{Quartic reciprocity}
We recall the main definitions and results of quartic reciprocity; see Chapter 6 of \cite{LEM00} for a longer exposition.

The Gaussian integers $\ZZ[i]$ form a unique factorization domain, with units being $\{1, i, -1, -i\}$. For $\alpha=a+bi\in\ZZ[i]$, denote the norm of $\alpha$ by $N(\alpha):=a^2+b^2$. We call $\alpha$ \emph{odd} if $N(\alpha)$ is odd, and \emph{even} otherwise. An even $\alpha$ is necessarily divisible by $1+i$.

\begin{definition}
	If $\alpha=a+bi\in\ZZ[i]$ is odd, call it \emph{primary} if $\alpha\equiv 1\pmod{2+2i}$. This is equivalent to $(a, b)\equiv (1, 0), (3, 2)\pmod{4}$.
\end{definition}

If $\alpha$ is odd, then exactly one \emph{associate} $\alpha$, $i\alpha$, $-\alpha$, $-i\alpha$ of $\alpha$ is primary.

\begin{definition}
	The quartic residue symbol $\quartic{\alpha}{\beta}$ takes in two coprime elements $\alpha,\beta\in\ZZ[i]$ with $\beta$ odd, and outputs a power of $i$. Let $\pi$ be an odd prime of $\ZZ[i]$. If $\alpha$ is coprime to $\pi$, define $\quartic{\alpha}{\pi}$ to be the unique power of $i$ that satisfies
	\[\quartic{\alpha}{\pi}\equiv\alpha^{\frac{N(\pi)-1}{4}}\pmod{\pi}.\]
	Extend the quartic residue symbol multiplicatively in the denominator:
 \[
 \quartic{\alpha}{ u \pi_1 \pi_2 \cdots \pi_n}
 = \quartic{\alpha}{u}
 \quartic{\alpha}{\pi_1}
 \quartic{\alpha}{\pi_2}
 \cdots
 \quartic{\alpha}{\pi_n},
 \]where $\quartic{\alpha}{u}=1$ for any unit $u\in\ZZ[i]$.
\end{definition}

The basic properties of this symbol and the statement of quartic reciprocity are summarized next, following \cite{LEM00}.

\begin{proposition}{{\cite[Propositions 4.1, 6.8]{LEM00}}}
\label{prop:quarticprops}
	The quartic residue symbol satisfies the following properties:
	\begin{enumerate}[label={\upshape\alph*)}]
		\item If $\alpha_1, \alpha_2 \in\ZZ[i]$ with $\alpha_1\alpha_2$ coprime to an odd $\beta \in \ZZ[i]$, then $\quartic{\alpha_1\alpha_2}{\beta}=\quartic{\alpha_1}{\beta}\quartic{\alpha_2}{\beta}$.
		\item If $\alpha_1, \alpha_2, \beta\in\ZZ[i]$ satisfy $\alpha_1\equiv \alpha_2\pmod{\beta}$, $\alpha_1$ and $\beta$ are coprime, and $\beta$ is odd, then $\quartic{\alpha_1}{\beta}=\quartic{\alpha_2}{\beta}$.
		\item If $a, b\in\ZZ$ are coprime integers with $b$ odd, then $\quartic{a}{b}=1$.
	\end{enumerate}
\end{proposition}

\begin{proposition}{{\cite[Theorem 6.9]{LEM00}}}
\label{prop:quarticreciprocity}
	Let $\alpha=a+bi$ be primary. Then
	\[
		\quartic{i}{\alpha}=i^{\frac{1-a}{2}}, \quad \quartic{-1}{\alpha}=i^b, \quad \quartic{1+i}{\alpha}=i^{\frac{a-b-b^2-1}{4}}, \quad \quartic{2}{\alpha}=i^{\frac{-b}{2}}.
	\]
	If $\beta\in\ZZ[i]$ is relatively prime to $\alpha$ and primary, then
	\[\quartic{\alpha}{\beta}=(-1)^{\frac{N(\alpha)-1}{4}\frac{N(\beta)-1}{4}}\quartic{\beta}{\alpha}.\]
	In particular, if either $\alpha$ or $\beta$ is an odd primary integer, then $\quartic{\alpha}{\beta}=\quartic{\beta}{\alpha}$.
\end{proposition}

\subsection{Definition of $\chi_4$}\label{sec:chi4def}
Let $\cir$ be a circle of curvature $n$. As we have seen, it is associated to $[f_{\cir}]$, a $\PGL(2, \ZZ)-$equivalence class of quadratic forms of discriminant $-4n^2$. By using the bijection between quadratic forms and fractional ideals, we obtain from $f_\cir$ a unique homothety class of lattices $[\Lambda]$, for some $\Lambda \subseteq \QQ[i]$.  
Up to multiplication by a unit, there is\footnote{The cited proposition is more general, but it is only in the case of class number one that this is true of every homothety class.} a unique representative $\Lambda_\cir$ of $[\Lambda]$ that lies inside $\ZZ[i]$ with covolume $n$ and conductor $n$ (definition below) \cite[Proposition 5.1]{KS18}.   
The values of $f_\cir$ are exactly the norms of the elements of $\Lambda_\cir$.   
By considering the elements of $\Lambda_\cir$, instead of only their norms, we can recover finer information than just quadratic residuosity:  we can access quartic residuosity.  This is the key insight in defining $\chi_4$.

In general, any rank-$2$ lattice $\Lambda \subseteq \ZZ[i]$ has an order, $\operatorname{ord}(\Lambda):= \{ \lambda \in \ZZ[i] : \lambda \Lambda \subseteq \Lambda \}$, which is an order of $\QQ(i)$, not necessarily maximal.  By the \emph{conductor of $\Lambda$}, we will mean the conductor of this order, which is a positive integer.  The following can be proven by observing that $\Lambda$ is locally principal, or by choosing a Hermite basis for $\Lambda$, which will have one element in $\operatorname{ord}(\Lambda)$. 

\begin{proposition}{{\cite[Proposition 6.3]{KS18}}}
\label{prop:slambdasingleton}
Let $n$ be a positive integer. 
	Let $\Lambda\subseteq\ZZ[i]$ be a lattice of covolume $n$ and conductor $n$, and define
	\[S_{\Lambda}:=\{\beta\in\Lambda:\beta\text{ is coprime to $n$}\}.\]
	Then the image of $S_{\Lambda}$ in $(\ZZ[i]/n\ZZ[i])^{\times}/(\ZZ/n\ZZ)^{\times}$ consists of a single element.
\end{proposition}

\begin{definition}\label{def:chifour}
	Let $\cpack$ be of type $(6, 1)$ or $(6, 17)$, and let $\cir$ be a circle of non-zero curvature $n$ in $\cpack$, necessarily satisfying $n\equiv 0, 1, 4\pmod{8}$. Suppose $\cir$ corresponds to a lattice $\Lambda_{\cir}\subseteq\ZZ[i]$ of covolume $n$ and conductor $n$. Let $\beta=a+bi\in S_{\Lambda_{\cir}}\cup S_{i\Lambda_{\cir}}$, where $\beta$ is chosen to be primary if $n$ is even.  Write $n=2^en'$ where $n'$ is odd. Define $\chi_4(\cir)$ as
	\[
	\chi_4(\cir):=\begin{cases}
		(-1)^{be/4}\quartic{\beta}{n'} & \text{if $n\equiv 0\pmod{8}$;}\\
		\quartic{\beta}{n} & \text{if $n\equiv 1\pmod{8}$;}\\
		\kron{-1}{n'}\quartic{\beta}{n'} & \text{if $n\equiv 4\pmod{8}$.}\\
	\end{cases}
	\]
\end{definition}

\begin{proposition}\label{prop:betaindependent}
	There exists a choice of $\beta$ satisfying all requirements, and $\chi_4$ is well-defined, independent of this choice, and lies in $\{1,i,-1,-i\}$.
\end{proposition}
\begin{proof}
	First, the set $S_{\Lambda_{\cir}}\cup S_{i\Lambda_{\cir}}$ is uniquely determined by $\cir$. If $n$ is odd and $\beta, \beta'$ are two choices, then by Proposition \ref{prop:slambdasingleton} we have $\beta'=i^k(u\beta+\delta)$ for $k=0, 1$, an integer $u$ coprime to $n$, and $\delta\in n\ZZ[i]$.
	Using Propositions \ref{prop:quarticprops} and \ref{prop:quarticreciprocity} (and recalling that $n\equiv 1\pmod{8}$), we compute
	\begin{equation}\label{eqn:betaindependant}
		\quartic{\beta'}{n}=\quartic{i}{n}^k\quartic{u\beta+\delta}{n}=\quartic{u\beta}{n}=\quartic{u}{n}\quartic{\beta}{n}=\quartic{\beta}{n},
	\end{equation}
	
	Next, assume $n$ is even, hence a multiple of $4$. Pick an arbitrary $\beta\in S_{\Lambda_{\cir}}$, which is necessarily odd. By replacing it with an associate, we can assume it is primary, which proves that a choice of $\beta$ is possible.
	
	As before, assume that two valid choices are $\beta, \beta'$, so that $\beta'=i^k(u\beta+\delta)$ for some integer $k=0, 1$, integer $u$ coprime to $n$, and $\delta\in n\ZZ[i]$. Also write $\beta=a+bi$ and $\beta'=a'+b'i$.
	
	If $k=1$, then $b$ and $b'$ have opposite parity, a contradiction to them being primary. Therefore $k=0$, so $\beta'=u\beta+\delta$. In particular, as $n'$ is an odd integer, the analogous computation to Equation \ref{eqn:betaindependant} still holds, so $\quartic{\beta'}{n'}=\quartic{\beta}{n'}$. It remains to check that the extra factors in the definition of $\chi_4$ in the even case are also independent.
	
	If $n\equiv 4\pmod{8}$, the extra factor is $\kron{-1}{n'}$, which does not depend on $\beta$. 
 
 	If $n\equiv 0\pmod{8}$, the extra factor is $(-1)^{be/4}$, which depends on $b$.  We must verify that $b\equiv 0\pmod{4}$, so that the exponent of $be/4$ is integral. Assuming this is true, and using that $8\mid \delta$ and $u$ is odd, we have $b'\equiv ub\equiv b\pmod{8}$, which completes the proof.
	
	To prove that $b \equiv 0 \pmod{4}$, note that $a^2+b^2=N(\beta)$ is a value properly represented by $f_{\cir}(x, y)$. If $b\not\equiv 0\pmod{4}$, then $a^2+b^2\equiv 5\pmod{8}$, hence $f_{\cir}(x, y) - n$ properly represents a number that is $5-0\equiv 5\pmod{8}$, i.e. $\cpack$ contains a curvature of this form. However, we are in a packing of type either $(6, 1)$ or $(6, 17)$, where all odd curvatures are $1\pmod{8}$, a contradiction.
\end{proof}

\subsection{Propagation of $\chi_4$}
Assume $\cpack$ is of type $(6, 1)$ or $(6, 17)$. In order to relate the $\chi_4$ values of tangent circles, we need a value of $\beta$ that works for both.

\begin{proposition}\label{prop:betatangent}
	Let $\cir_1,\cir_2\in\cpack$ be tangent circles of coprime curvatures $n_1, n_2$. Then there exists  $\beta\in\ZZ[i]$ such that
	 $N(\beta)=n_1+n_2$ and  $\beta$ is a valid choice in Definition \ref{def:chifour} for both $\cir_1$ and $\cir_2$.
\end{proposition}
\begin{proof}
The orbit of the extended real line $\widehat{\RR}$ under the M\"obius transformations $\PSL(2, \ZZ[i])$ is a collection of circles in the extended complex plan $\widehat{\CC}$, called the \emph{Schmidt arrangement} \cite[Definition 1.1]{KS18}.  After a scaling by $2$, the curvatures of all circles in the arrangement are integers \cite[Proposition 3.7]{KS18}.  The scaled Schmidt arrangement coincides with the \emph{Apollonian super-packing} of \cite{GLMWY06}, and contains, up to the symmetries of the arrangement, exactly one copy of every primitive Apollonian circle packing (\cite[Theorem 6.2]{GLMWY06} and \cite[Theorem 1.3]{KS18_Bi}).  In this way, the packing $\cpack$ can be given a definite location in $\widehat{\CC}$, and its tangency points can be described as elements of $\QQ(i) \subseteq \CC$.  

Let $\cir \in \cpack$ of curvature $n$ inside the Schmidt arrangement be the image of $\widehat{\RR}$ under the M\"obius transformation $\sm{\alpha}{\gamma}{\beta}{\delta}$.  Proposition 4.6 of \cite{KS18} describes the set of tangency points of $\cir$ with the circles of $\cpack$ which it touches, namely 
\[
\left\{ \frac{\eta}{\mu} : \lmvec{\eta}{\mu} \in \lmvec{\alpha}{\beta} \ZZ + \lmvec{\gamma}{\delta}\ZZ \right\}.
\]
The lattice of denominators $\beta \ZZ + \delta \ZZ$ then coincides with $\Lambda_\cir$ \cite[Theorem 4.7]{KS18}.  In particular, we can write $f_\cir(x, y) = N(x \beta + y \delta)$, so that the circle tangent to $\cir$ at point $x \smvec{\alpha}{\beta} + y \smvec{\gamma}{\delta}$ has curvature $N(x \beta + y \delta) - n$ \cite[Proposition 4.6]{KS18}.

Hence, if we consider tangent circles $\cir_1$ and $\cir_2$ of curvatures $n_1$ and $n_2$ respectively, then we obtain two lattices $\Lambda_{\cir_1}$ and $\Lambda_{\cir_2}$ which share an element $\beta$, namely the denominator of the shared tangency point.  In particular, $N(\beta) = n_1 + n_2$.

Finally, the lattice $\Lambda_\cir$ associated to a circle $\cir$ is defined only up to unit multiple, which allows for us to arrange for $\beta$ to be primary if necessary.
\end{proof}

\begin{proposition}\label{prop:chi4onestep}
	Let $\cir_1,\cir_2\in\cpack$ be tangent circles of coprime curvature. Then $\chi_4(\cir_1)=\chi_4(\cir_2)$.
\end{proposition}
\begin{proof}
Let $n_1$ and $n_2$ be the curvatures of $\cir_1$ and $\cir_2$ respectively.
Since $\cpack$ has type $(6,1)$ or $(6,17)$, $n_1,n_2 \equiv 0,1$, or $4 \pmod{8}$.
	Take a $\beta$ as promised by Proposition \ref{prop:betatangent}, and assume that $n_1$ is odd. If $n_2$ is also odd, then $N(\beta)=n_1+n_2\equiv 2\pmod{8}$, hence $\beta=(1+i)\beta'$, with $\beta'$ odd. By replacing $\beta$ with an associate, we can assume that $\beta'=a+bi$ is primary. We compute
	\[
	\chi_4(\cir_1)=\quartic{\beta}{n_1}=\quartic{1+i}{n_1}\quartic{\beta'}{n_1}=i^{\frac{n_1-1}{4}}\quartic{n_1}{\beta'}.
	\]
	As $n_1+n_2=N(\beta)=\beta\overline{\beta}$, we have $n_1\equiv -n_2\pmod{\beta'}$.  Thus
	\[
	\chi_4(\cir_1)=i^{\frac{n_1-1}{4}}\quartic{-n_2}{\beta'}=i^{\frac{n_1-1}{4}}\quartic{-1}{\beta'}\quartic{n_2}{\beta'}=i^{\frac{n_1-1}{4}}i^b\quartic{\beta'}{n_2}=i^{\frac{n_1-1}{4}}i^bi^{-\frac{n_2-1}{4}}\chi_4(\cir_2).
	\]
	In order to conclude that $\chi_4(\cir_1)=\chi_4(\cir_2)$, we must have $\frac{n_1-1}{4}+b-\frac{n_2-1}{4}\equiv 0\pmod{4}$, i.e.
	\begin{equation}\label{eqn:chi4onemod8}
		n_1-n_2+4b\equiv 0\pmod{16}.
	\end{equation}
	If $n_1\equiv n_2\pmod{16}$, then $2(a^2+b^2)=n_1+n_2\equiv 2\pmod{16}$, hence $a^2+b^2\equiv 1\pmod{8}$. In particular, $4\mid b$ (recall that $\beta'$ is primary), and Equation \ref{eqn:chi4onemod8} follows. Otherwise, $n_1\equiv n_2 + 8\pmod{16}$, so $2(a^2+b^2)\equiv 10\pmod{16}$, and $a^2+b^2\equiv 5\pmod{8}$. This implies that $b\equiv 2\pmod{4}$, and again Equation \ref{eqn:chi4onemod8} is true.
	
	If $n_2$ is even, then $\beta=a+bi$ is primary by assumption. We compute
	\[
	\chi_4(\cir_1)=\quartic{\beta}{n_1}=\quartic{n_1}{\beta}=\quartic{-n_2}{\beta}=\quartic{n_2/4}{\beta},
	\]
	where the last equality follows from $\quartic{-4}{\beta}=\quartic{1+i}{\beta}^4=1$. We now separate into the cases $n_2\equiv 0$ or $4 \pmod{8}$.
	
	If $n_2\equiv 0\mod{8}$, write $n_2=2^en_2'$ with $n_2'$ odd, and 
 	\[
	\chi_4(\cir_2)=(-1)^{be/4}\quartic{\beta}{n_2'}=(-1)^{be/4}\quartic{\pm n_2'}{\beta}=(-1)^{be/4}\quartic{\mp 1}{\beta}\quartic{-n_2'}{\beta}=(-1)^{be/4}\quartic{\mp 1}{\beta}\quartic{2^{e}}{\beta}^{-1}\chi_4(\cir_1),
	\]
	where the sign of the $\pm$ depends on $n_2'$ modulo $4$. As in the proof of Proposition \ref{prop:betaindependent}, we have $4\mid b$, hence $\quartic{-1 }{\beta}=i^b=1$, so the $\pm$ sign does not matter. We also compute 
	\[\quartic{2^{e}}{\beta}^{-1}=\left(i^{-b/2}\right)^{-e}=(-1)^{be/4}=(-1)^{-be/4},\]
	hence the terms cancel and $\chi_4(\cir_2)=\chi_4(\cir_1)$.
	
	Finally, assume $n_2\equiv 4\pmod{8}$, so $n_2'=n_2/4$. If $n_2'\equiv 1\pmod{4}$, then
	\[
	\chi_4(\cir_2)=\kron{-1}{n_2'}\quartic{\beta}{n_2'}=\quartic{n_2'}{\beta}=\chi_4(\cir_1),
	\]
	as desired. Otherwise, $n_2'\equiv 3\pmod{4}$ and
	\[
	\chi_4(\cir_2)=\kron{-1}{n_2'}\quartic{\beta}{n_2'}=-\quartic{-n_2'}{\beta}=-\quartic{-1}{\beta}\chi_4(\cir_1)=-i^{-b}\chi_4(\cir_1).
	\]
	Since $a^2+b^2=n_1+n_2\equiv 5\pmod{8}$, we have $b\equiv 2\pmod{4}$, so $-i^b=1$, completing the proof.
\end{proof}

Corollary \ref{cor:coprimepath} and Proposition \ref{prop:chi4onestep} combine to give the following corollary.

\begin{corollary}
	The value of $\chi_4$ is constant across all circles in a packing $\cpack$ of type $(6, 1)$ or $(6, 17)$. Denote this value by $\chi_4(\cpack)$.
\end{corollary}

\subsection{Quartic obstructions}

\begin{proposition}
	Then the following quartic obstructions occur, as a function of type and $\chi_4(\cpack)$: 
	\tablestretch{
        \begin{tabular}{|c|c|c|} 
            \hline
            Type    & $\chi_4(\cpack)$ & Quartic obstructions\\ \hline
            (6, 1)  & $1$              & \\ 
                    & $-1, i, -i$      & $n^4, 4n^4, 9n^4, 36n^4$\\ \hline
            (6, 17) & $1$              & $9n^4, 36n^4$\\
                    & $-1$             & $n^4, 4n^4$\\
                    & $i, -i$          & $n^4, 4n^4, 9n^4, 36n^4$\\ \hline
        \end{tabular}
    }
\end{proposition}

\begin{proof}
	Let $u\in\{1, 4, 9, 36\}$, and assume that a circle $\cir$ of curvature $n=uw^4$ appears in $\cpack$ for some positive integer $w$. Let $n_2$ be the curvature of a circle $\cir'$ tangent to $\cir$ that is coprime to $n$. Let $\beta$ be chosen as in Proposition~\ref{prop:betatangent} for the circles $\cir$ and $\cir'$.
	
	If $n\equiv 0\pmod{8}$, then $2\mid w$, hence $n=2^en'$ with $n'$ odd and $2\mid e$. Thus
	\[\chi_4(\cir)=(-1)^{be/4}\quartic{\beta}{n'}=\begin{cases}
		\quartic{\beta}{1} & \text{if $u=1, 4$;}\\[7pt]
		\quartic{\beta}{3}^2 & \text{if $u=9, 36$.}\\
	\end{cases}
	\]
	Clearly $\quartic{\beta}{1}=1$, which gives the result for $u=1,4$. For $u=9, 36$, we have $n\equiv 0\pmod{24}$. Let $\beta=a+bi$. Since $3$ is prime in $\ZZ[i]$,
	\[
	\quartic{\beta}{3}^2\equiv\beta^4\equiv (a+bi)^4\equiv a^4+b^4+(a^3b-ab^3)i\pmod{3}.
	\]
	Then $a^2+b^2= n+n_2\equiv n_2\pmod{3}$. If the type of $\cpack$ is $(6, 1)$, then $a^2+b^2\equiv 1\pmod{3}$, so exactly one of $(a, b)$ is $0\pmod{3}$. In either case, $a^4+b^4+(a^3b-ab^3)i\equiv 1\pmod{3}$, so $\chi_4(\cir)=1$. If the type is $(6, 17)$, then $a^2+b^2\equiv 2\pmod{3}$, so $a^2\equiv b^2\equiv 1\pmod{3}$. Thus
	\[a^4+b^4+(a^3b-ab^3)i\equiv \left(a^2\right)^2+\left(b^2\right)^2+ab(a^2-b^2)i\equiv 2\equiv -1\pmod{3},\]
	so $\chi_4(\cir)=-1$, again agreeing with the table.
	
	Next, assume $n\equiv 1\pmod{8}$. If $u=1$, we have $\chi_4(\cir)=\quartic{\beta}{w^4}=1$. Otherwise $u=9$, and in fact $n\equiv 9\pmod{24}$. Then $\chi_4(\cir)=\quartic{\beta}{3^2 w^4}=\quartic{\beta}{3}^2$, from whence the analysis proceeds exactly as for $n\equiv 0 \pmod{8}$.
	
	Finally, take $n\equiv 4\pmod{8}$, which allows $u=4, 36$. If $u=4$, then $n'$ is an odd fourth power, so $\chi_4(\cir)=\kron{-1}{n'}\quartic{\beta}{n'}=1$. If $u=36$, then $n'=9t^4$, so $\chi_4(\cir)=\quartic{\beta}{3}^2$; proceed as for $n\equiv 0 \pmod{8}$.
\end{proof}

There is a nice relationship between $\chi_4(\cpack)$ and $\chi_2(\cpack)$.

\begin{proposition}
	$\chi_4(\cpack)^2=\chi_2(\cpack)$.
\end{proposition}
\begin{proof}
	Let $\cir$ be a circle of odd curvature $n$ in $\cpack$.  Choose a circle $\cir'$ tangent to $\cir$ of coprime curvature $n_2$, and choose $\beta$ as in Proposition~\ref{prop:betatangent} for circles $\cir$ and $\cir'$.  Then $\chi_4(\cir)^2=\quartic{\beta}{n}^2$. Let $\kron{\cdot}{\cdot}$ also denote the quadratic residue symbol for $\ZZ[i]$; it follows that $\chi_4(\cir)^2=\kron{\beta}{n}$. By \cite[Proposition 4.2iii)]{LEM00}, we have $\kron{\beta}{n}=\kron{N(\beta)}{n}$, with the second Kronecker symbol taken over $\ZZ$.  Since $N(\beta)=n+n_2$, we can take $\rho(f_{\cir})=n+n_2$, proving that $\chi_4(\cir)^2=\chi_2(\cir)$. As $\chi_2$ and $\chi_4$ are constant across $\cpack$, the result follows.
\end{proof}

\section{Computations}\label{sec:computations}

In order to support Conjecture~\ref{conj:newlocalglobal}, code to compute the missing curvatures was written with a combination of C and PARI/GP \cite{PARI}. This code (alongside other methods to compute with Apollonian circle packings) can be found in the GitHub repository \cite{GHapollonian}.

Files containing the sporadic sets $S_{\cpack}(N):=S_{\cpack}\bigcap[1, N]$ for many small root quadruples can be found in the GitHub repository \cite{GHmissingcurvatures}. In particular, for each of the 14 types listed in Theorem \ref{thm:mainthm}, we took three small root quadruples, and computed the sporadic curvatures up to a bound $N$ in the range $[10^{10}, 10^{12}]$. The bound $N$ was chosen so that the ratio of $N$ to the largest sporadic curvature found exceeds $10$, providing good evidence towards Conjecture \ref{conj:newlocalglobal}. These results are summarized in the appendix, Tables \ref{table:conjecturedata1} and \ref{table:conjecturedata2}.

\begin{remark}
	There are only five pairs of residue class and packing for which it appears that \emph{every single} positive residue in that class appears. They are the following:
	\begin{itemize}
		\item $(-3, 5, 8, 8)$ and $5\pmod{24}$;
		\item $(-3, 4, 12, 13)$ and $13\pmod{24}$;
		\item $(-1, 2, 2, 3)$ and $11, 14, 23\pmod{24}$.
	\end{itemize}
	Near misses are
	\begin{itemize}
		\item $(0, 0, 1, 1)$ and $1\pmod{24}$, which only misses the curvature $241$ up to $10^{10}$;
		\item $(-1, 2, 2, 3)$ and $2\pmod{24}$, which only misses the curvature $13154$ up to $10^{10}$.
	\end{itemize}
\end{remark}

\begin{remark}
    An intrepid observer of the raw sporadic sets may remark that, toward the tail end, the sporadic curvatures are disproportionately multiples of $5$.  In fact, they generally prefer prime divisors which are $1 \pmod{4}$.  We speculate that this is another local phenomenon:  a result of certain symmetries of the distribution of curvatures in the orbit of quadruples modulo $p \equiv 1 \pmod{4}$ (similar to \cite[Figures 3 and 4]{FS11}).
    \end{remark}

\begin{aremark}[not in published version]
 One particularly visually appealing way to observe the reciprocity obstructions is to plot the successive differences of the exceptional set.  Since quadratic and quartic sequences have patternful successive differences, even a union of quadratic and quartic sequences reveals a prominent visual pattern once the sporadic set begins to thin out. See Figure~\ref{fig:differences}, in the appendix.
\end{aremark}

\clearpage

\appendix
\section{Additional tables and figures}

\setcounter{table}{0}
\renewcommand*{\thetable}{\Alph{table}}

\tablestretch{
\begin{table}[htb]
	\caption{(not in published version) $S_{\cpack}(N)$ for small packings: part 1.}\label{table:conjecturedata1}
	\begin{tabular}{|l|c|c|c|c|c|c|c|} 
		\hline
		Packing             & Type             & Quadratic    & Quartic & $N$         & $|S_{\cpack}(N)|$ &  $\max(S_{\cpack}(N))$ & $\approx\frac{N}{\max(S_{\cpack}(N))}$\\ \hline
		$(0, 0, 1, 1)$      & $(6, 1, 1, 1)$   &              &         & $10^{10}$   & 215               & 1199820                & 8334.58\\
		$(-12, 16, 49, 49)$ &                  &              &         & $10^{11}$   & 275276            & 5542869468             & 18.04\\
		$(-20, 36, 49, 49)$ &                  &              &         & $10^{12}$   & 2014815           & 55912619880            & 17.89\\ \hline
		$(-8, 12, 25, 25)$  & $(6, 1, 1, -1)$  &              & $n^4, 4n^4,$ & $10^{10}$ & 47070          & 517280220              & 19.33\\
		$(-12, 25, 25, 28)$ &                  &              & $9n^4, 36n^4$ & $10^{11}$ & 238268        & 5919707820             & 16.89\\
		$(-15, 24, 40, 49)$ &                  &              &         & $2\cdot 10^{11}$ & 639149       & 12692531688            & 15.75\\ \hline
		$(-15, 28, 33, 40)$ & $(6, 1, -1)$     & $n^2, 2n^2,$ &         & $10^{10}$   & 80472             & 820523160              & 12.19\\
		$(-20, 33, 52, 57)$ &                  & $3n^2, 6n^2$ &         & $10^{11}$   & 240230            & 4127189100             & 24.23\\
		$(-23, 40, 57, 60)$ &                  &              &         & $10^{11}$   & 392800            & 8689511520             & 11.51\\ \hline
		$(-4, 5, 20, 21)$   & $(6, 5, 1)$      & $2n^2, 3n^2$ &         & $10^{10}$   & 3659              & 32084460               & 311.68\\
		$(-16, 29, 36, 45)$ &                  &              &         & $10^{10}$   & 80256             & 927211800              & 10.79\\
		$(-19, 36, 44, 45)$ &                  &              &         & $10^{11}$   & 177902            & 3603790320             & 27.75\\ \hline
		$(-3, 5, 8, 8)$     & $(6, 5, -1)$     & $n^2, 6n^2$  &         & $10^{10}$   & 676               & 3122880                & 3202.17\\
		$(-12, 21, 29, 32)$ &                  &              &         & $10^{10}$   & 30347             & 312225420              & 32.03\\
		$(-19, 32, 48, 53)$ &                  &              &         & $2.5\cdot 10^{10}$   & 168264   & 2286209460             & 10.94\\ \hline
		$(-3, 4, 12, 13)$   & $(6, 13, 1)$     & $2n^2, 6n^2$ &         & $10^{10}$   & 731               & 7354464                & 1359.72\\
		$(-12, 21, 28, 37)$ &                  &              &         & $10^{11}$   & 234386            & 3470731680             & 28.81\\
		$(-11, 16, 36, 37)$ &                  &              &         & $10^{10}$   & 20748             & 226988340              & 44.06\\ \hline
		$(-8, 13, 21, 24)$  & $(6, 13, -1)$    & $n^2, 3n^2$  &         & $10^{10}$   & 5273              & 45348900               & 220.51\\
		$(-11, 21, 24, 28)$ &                  &              &         & $10^{10}$   & 21003             & 176441136              & 56.68\\
		$(-20, 37, 45, 52)$ &                  &              &         & $10^{11}$   & 229356            & 4079861484             & 24.51\\ \hline
		$(-16, 32, 33, 41)$ & $(6, 17, 1, 1)$  & $3n^2, 6n^2$ & $9n^4, 36n^4$ & $10^{10}$ & 81777         & 841440840              & 11.88\\
		$(-7, 8, 56, 57)$   &                  &              &         & $10^{10}$   & 55057             & 595231740              & 16.80\\
		$(-16, 20, 81, 81)$ &                  &              &         & $10^{12}$   & 1075024           & 26983035480            & 37.06\\ \hline
		$(-4, 8, 9, 9)$     & $(6, 17, 1, -1)$ & $3n^2, 6n^2$ & $n^4, 4n^4$ & $10^{10}$ & 2057            & 10742460               & 930.89\\
		$(-7, 9, 32, 32)$   &                  &              &         & $10^{10}$   & 34916             & 367956840              & 27.18\\
		$(-15, 32, 32, 33)$ &                  &              &         & $10^{11}$   & 585942       	  & 8505627180             & 11.76\\ \hline
		$(-7, 12, 17, 20)$  & $(6, 17, -1)$    & $n^2, 2n^2$  &         & $10^{10}$   & 3744              & 17141220               & 583.39\\
		$(-12, 17, 41, 44)$ &                  &              &         & $10^{10}$   & 31851             & 270186456              & 37.01\\
		$(-15, 24, 41, 44)$ &                  &              &         & $10^{10}$   & 80106             & 803343900              & 12.45\\ \hline
	\end{tabular}
\end{table}
}

\clearpage

\tablestretch{
\begin{table}[htb]
	\caption{(not in published version) $S_{\cpack}(N)$ for small packings: part 2.}\label{table:conjecturedata2}
	\begin{tabular}{|l|c|c|c|c|c|c|c|} 
		\hline
		Packing             & Type            & Quadratic    & Quartic & $N$         & $|S_{\cpack}(N)|$ &  $\max(S_{\cpack}(N))$ & $\approx\frac{N}{\max(S_{\cpack}(N))}$ \\ \hline
		$(-5, 7, 18, 18)$   & $(8, 7, 1)$     & $3n^2, 6n^2$ &         & $10^{10}$   & 16417             & 86709570               & 115.33\\
		$(-6, 10, 15, 19)$  &                 &              &         & $10^{10}$   & 24305             & 133977255              & 74.64\\
		$(-9, 18, 19, 22)$  &                 &              &         & $10^{10}$   & 14866             & 82815750               & 120.75\\ \hline
		$(-2, 3, 6, 7)$     & $(8, 7, -1)$    & $18n^2$      &         & $10^{10}$   & 236               & 429039                 & 23307.90\\
		$(-5, 6, 30, 31)$   &                 &              &         & $10^{10}$   & 19695             & 97583070               & 102.48\\
		$(-14, 27, 31, 34)$ &                 &              &         & $2\cdot 10^{10}$ & 99294        & 1643827935             & 12.17\\ \hline
		$(-1, 2, 2, 3)$     & $(8, 11, 1)$    &              &         & $10^{10}$   & 61                & 97287                  & 102788.66\\
		$(-9, 14, 26, 27)$  &                 &              &         & $10^{10}$   & 17949             & 85926675               & 116.38\\
		$(-10, 18, 23, 27)$ &                 &              &         & $10^{10}$   & 25944             & 124625694              & 80.24\\ \hline
		$(-6, 11, 14, 15)$  & $(8, 11, -1)$   & $2n^2, 3n^2,$ &        & $10^{10}$   & 3381              & 20149335               & 496.29\\
		$(-10, 14, 35, 39)$ &                 & $6n^2$       &         & $4\cdot10^{10}$ & 256228        & 2934238515             & 13.63\\
		$(-13, 23, 30, 38)$ &                 &              &         & $10^{10}$   & 71341             & 598107510              & 16.72\\ \hline
	\end{tabular}
\end{table}
}

\setcounter{figure}{0}
\renewcommand*{\thefigure}{\Alph{figure}}

\begin{figure}[htb]
	\includegraphics[width=0.6\textwidth]{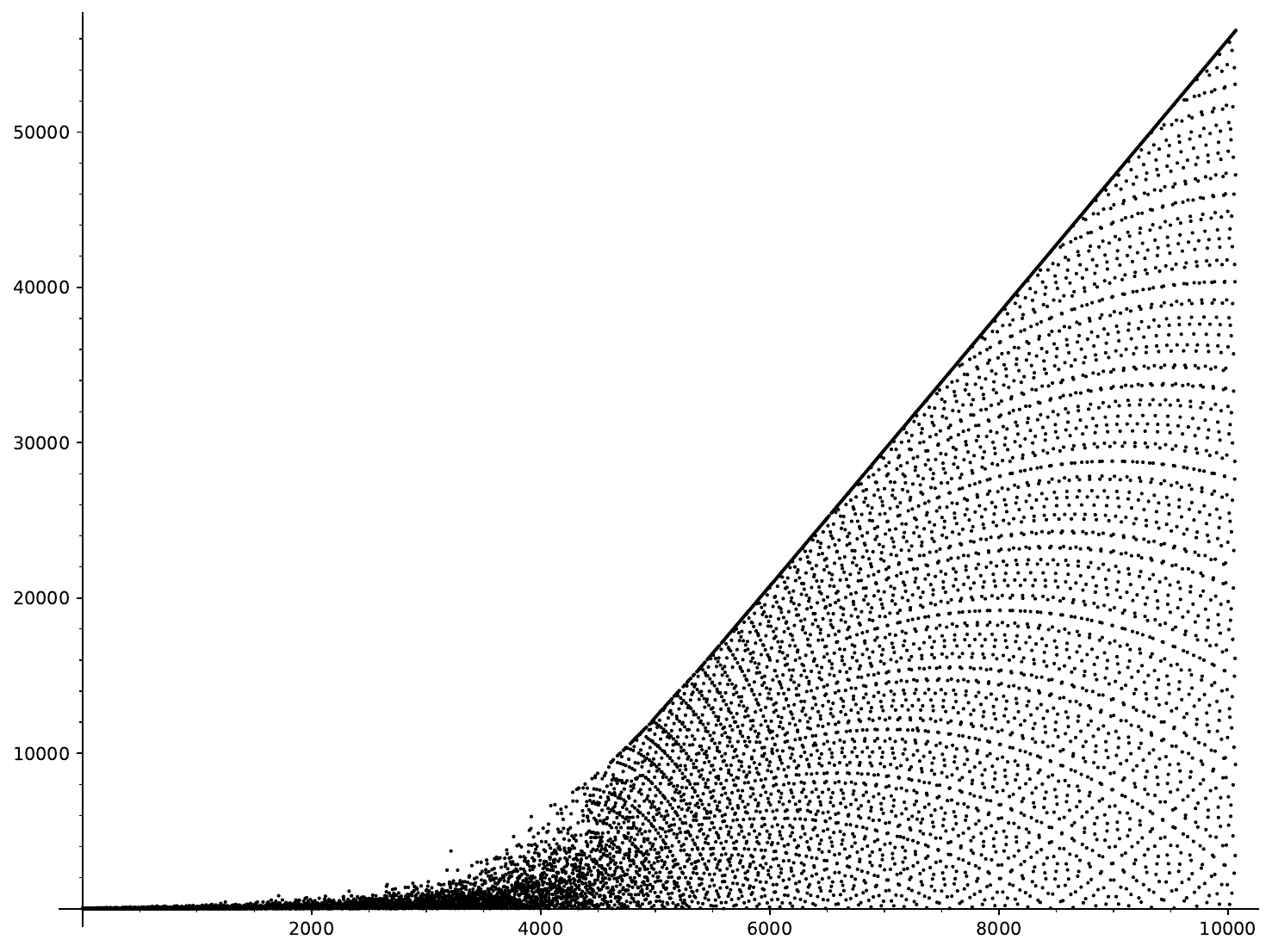}
	\caption{(not in published version) Successive differences of missing curvatures in the packing $(-4,5,20,21)$ of type (6,5,1).  Around the 5000th missing curvature, the quadratic families $2n^2$ and $3n^2$ begin to predominate (the sporadic set has $3659$ elements $< 10^{10}$, and they continue to occur even past the region of this graph, but very sparsely.).}
	\label{fig:differences}
\end{figure}

\clearpage
	
\bibliographystyle{alpha}
\bibliography{references}

\newcommand{\etalchar}[1]{$^{#1}$}
\begin{thebibliography}{GLM{\etalchar{+}}06}

\bibitem[BK14]{BK14}
Jean Bourgain and Alex Kontorovich.
\newblock On the local-global conjecture for integral {A}pollonian gaskets.
\newblock {\em Invent. Math.}, 196(3):589--650, 2014.
\newblock With an appendix by P\'{e}ter P. Varj\'{u}.

\bibitem[Bue89]{BU89}
Duncan~A. Buell.
\newblock {\em Binary quadratic forms}.
\newblock Springer-Verlag, New York, 1989.
\newblock Classical theory and modern computations.

\bibitem[Coh93]{Cohen93}
Henri Cohen.
\newblock {\em A course in computational algebraic number theory}, volume 138
  of {\em Graduate Texts in Mathematics}.
\newblock Springer-Verlag, Berlin, 1993.

\bibitem[FS11]{FS11}
Elena Fuchs and Katherine Sanden.
\newblock Some experiments with integral {A}pollonian circle packings.
\newblock {\em Exp. Math.}, 20(4):380--399, 2011.

\bibitem[FSZ19]{FSZ19}
Elena Fuchs, Katherine~E. Stange, and Xin Zhang.
\newblock Local-global principles in circle packings.
\newblock {\em Compos. Math.}, 155(6):1118--1170, 2019.

\bibitem[Fuc10]{Fuchs10}
Elena Fuchs.
\newblock {\em Arithmetic properties of {A}pollonian circle packings}.
\newblock ProQuest LLC, Ann Arbor, MI, 2010.
\newblock Thesis (Ph.D.)--Princeton University.

\bibitem[Fuc11]{Fuchs11}
Elena Fuchs.
\newblock Strong approximation in the {A}pollonian group.
\newblock {\em J. Number Theory}, 131(12):2282--2302, 2011.

\bibitem[GLM{\etalchar{+}}03]{GLMWY02}
Ronald~L. Graham, Jeffrey~C. Lagarias, Colin~L. Mallows, Allan~R. Wilks, and
  Catherine~H. Yan.
\newblock Apollonian {C}ircle {P}ackings: {N}umber {T}heory.
\newblock {\em J. Number Theory}, 100(1):1--45, 2003.

\bibitem[GLM{\etalchar{+}}06]{GLMWY06}
Ronald~L. Graham, Jeffrey~C. Lagarias, Colin~L. Mallows, Allan~R. Wilks, and
  Catherine~H. Yan.
\newblock Apollonian circle packings: geometry and group theory. {II}.
  {S}uper-{A}pollonian group and integral packings.
\newblock {\em Discrete Comput. Geom.}, 35(1):1--36, 2006.

\bibitem[KK23]{KK23}
Michael Kapovich and Alex Kontorovich.
\newblock On superintegral {K}leinian sphere packings, bugs, and arithmetic
  groups.
\newblock {\em J. Reine Angew. Math.}, 798:105--142, 2023.

\bibitem[Kon13]{Kont13}
Alex Kontorovich.
\newblock From {A}pollonius to {Z}aremba: local-global phenomena in thinorbits.
\newblock {\em Bull. Amer. Math. Soc. (N.S.)}, 50(2):187--228, 2013.

\bibitem[Kon19]{Kont19}
Alex Kontorovich.
\newblock The local-global principle for integral {S}oddy sphere packings.
\newblock {\em J. Mod. Dyn.}, 15:209--236, 2019.

\bibitem[Lem00]{LEM00}
Franz Lemmermeyer.
\newblock {\em Reciprocity laws}.
\newblock Springer Monographs in Mathematics. Springer-Verlag, Berlin, 2000.
\newblock From Euler to Eisenstein.

\bibitem[PAR23]{PARI}
The PARI~Group, Univ. Bordeaux.
\newblock {\em PARI/GP version \texttt{2.16.1}}, 2023.
\newblock available from \url{https://pari.math.u-bordeaux.fr/}.

\bibitem[Ric23a]{GHapollonian}
James Rickards.
\newblock Apollonian.
\newblock \url{https://github.com/JamesRickards-Canada/Apollonian}, 2023.

\bibitem[Ric23b]{GHmissingcurvatures}
James Rickards.
\newblock Apollonian {M}issing {C}urvatures.
\newblock
  \url{https://github.com/JamesRickards-Canada/Apollonian-Missing-Curvatures},
  2023.

\bibitem[Ric24]{JR23stairs}
James Rickards.
\newblock The {A}pollonian staircase.
\newblock {\em International Mathematics Research Notices}, 2024(2):1340--1372,
  2024.

\bibitem[RS24]{RickardsStangeTwo}
James Rickards and Katherine~E. Stange.
\newblock Reciprocity obstructions in semigroup orbits in
  {$\operatorname{SL}(2, \mathbb{Z})$}.
\newblock \url{https://arxiv.org/abs/2401.01860}, 2024.

\bibitem[Sar07]{Sar07}
Peter Sarnak.
\newblock Letter to {J}. {L}agarias.
\newblock \url{https://web.math.princeton.edu/sarnak/AppolonianPackings.pdf},
  2007.

\bibitem[Sar11]{SarnakMonthly}
Peter Sarnak.
\newblock Integral {A}pollonian packings.
\newblock {\em Amer. Math. Monthly}, 118(4):291--306, 2011.

\bibitem[Sta18a]{KS18_Bi}
Katherine~E. Stange.
\newblock The {A}pollonian structure of {B}ianchi groups.
\newblock {\em Trans. Amer. Math. Soc.}, 370(9):6169--6219, 2018.

\bibitem[Sta18b]{KS18}
Katherine~E. Stange.
\newblock Visualizing the arithmetic of imaginary quadratic fields.
\newblock {\em Int. Math. Res. Not. IMRN}, (12):3908--3938, 2018.

\end{thebibliography}
\end{document}